\documentclass[10pt,reqno]{amsart}

\usepackage{enumerate}
\usepackage{amsmath}
\usepackage{amssymb}
\usepackage{amsfonts}
\usepackage{amsthm}

\newtheorem{thm}{Theorem}[section]
\newtheorem{lem}[thm]{Lemma}
\newtheorem{prop}[thm]{Proposition}
\newtheorem{cor}[thm]{Corollary}

\theoremstyle{definition}
\newtheorem{dfn}[thm]{Definiton}

\theoremstyle{remark}

\newcommand{\set}[1]{\{#1\}}

\newcommand{\ga}{\gamma}
\newcommand{\de}{\delta}
\newcommand{\e}{\epsilon}
\newcommand{\N}{\mathbb{N}}
\newcommand{\Z}{\mathbb{Z}}

\newcommand{\CM}{\'Ciri\'c-Matkowski}

\newcommand{\m}{m}

\numberwithin{equation}{section}

\title[Fixed point theorems in generalized metric spaces]%
{Fixed point theorems of \'Ciri\'c-Matkowski type\\in generalized metric spaces}

\author[M. Abtahi]{Mortaza Abtahi}

\date{\today}

\begin{document}

\begin{abstract}
  A self-map $T$ of a $\nu$-generalized metric space $(X,d\,)$ is said to be a
  \'Ciri\'c-Matkowski contraction if $d(Tx,Ty)<d(x,y)$, for $x\neq y$, and,
  for every $\e>0$, there is $\de>0$ such that $d(x,y)<\de+\e$ implies $d(Tx,Ty)\leq \e$.
  In this paper, fixed point theorems for this kind of contractions of
  $\nu$-generalized metric spaces, are presented.
  Then, by replacing the distance function $d(x,y)$ with functions of the form $m(x,y)=d(x,y)+\ga\bigl(d(x,Tx)+d(y,Ty)\bigr)$, where $\ga>0$, results analogue to those
  due to P.D.\ Proiniv (Fixed point theorems in metric spaces, Nonlinear Anal. 46 (2006) 546--557)
  are obtained.
\end{abstract}

\maketitle

\section{Introduction}
\label{sec:intro}

Throughout the paper, the set of integers is denoted by $\Z$,
the set of nonnegative integers is denoted by $\Z^+$, and
the set of positive integers is denoted by $\N$.

Fixed point theory in metric spaces have many applications. It is natural that there have
been several attempts to extend it to a more general setting. One of these generalizations
was introduced by Branciari in 2000, where the triangle inequality was replaced by a so-called
\emph{quadrilateral inequality.} They introduced the concept of $\nu$-generalized metric spaces
as follows; see also
\cite{Alamri-Suzuki-Khan,Kadelburg-Radenovic-1,Kirk-Shahzad,Suzuki-Alamri-Khan}.

\begin{dfn}[Branciari \cite{Branciari}]
\label{dfn:nu-generalized-ms}
  Let $X$ be a nonvoid set and $d:X\times X\to[0,\infty)$ be a function. Let $\nu\in\N$.
  Then $(X,d\,)$ is called a \emph{$\nu$-generalized metric space} if the following hold:
  \begin{enumerate}[\quad$(1)$]
    \item $d(x,y)=0$ if and only if $x=y$, for every $x,y\in X$;

    \item $d(x,y)=d(y,x)$, for every $x,y\in X$;

    \item \label{item:nu-angle-inequality}
    $d(x,y) \leq d(x,u_1)+d(u_1,u_2)+\dotsb+d(u_\nu,y)$,
          for every set $\{x,u_1,\dotsc,u_\nu,y\}$ of $\nu+2$
          elements of $X$ that are all different.
  \end{enumerate}
\end{dfn}

Obviously, $(X,d\,)$ is a metric space if and only if it is a $1$-generalized metric space.
In \cite{Alamri-Suzuki-Khan}, the completeness of $\nu$-generalized metric spaces are discussed.
In \cite{Suzuki}, it is shown that not every generalized metric space has the compatible topology.

\begin{dfn}
  Let $(X,d\,)$ be a $\nu$-generalized metric space. Let $k\in\N$.
  A sequence $\{x_n\}$ in $X$ is said to be \emph{$k$-Cauchy} if
    \begin{equation}\label{eqn:k-Cauchy}
      \lim_{n\to\infty} \sup\set{d(x_n,x_{n+1+mk}):m\in \Z^+}=0.
    \end{equation}
  The sequence $\{x_n\}$ is said to be \emph{Cauchy} if it is $1$-Cauchy.
\end{dfn}

The concept of Cauchy sequences in $\nu$-generalized metric spaces are studied
in \cite{Alamri-Suzuki-Khan,Suzuki-Alamri-Khan}; see also \cite{Branciari}.

\begin{prop}[\cite{Alamri-Suzuki-Khan} and \cite{Suzuki-Alamri-Khan}]
\label{prop:nu-Cauchy-is-Cauchy}
  Let $(X,d\,)$ be a $\nu$-generalized metric space and let $\{x_n\}$ be a sequence
  in $X$ such that $x_n\ (n\in\N)$ are all different. Suppose $\{x_n\}$ is
  $\nu$-Cauchy. If $\nu$ is odd, or if $\nu$ is even and $d(x_n,x_{n+2})\to0$,
  then $\{x_n\}$ is Cauchy.
\end{prop}

A sequence $\{x_n\}$ in a $\nu$-generalized metric space $(X,d\,)$ is said
to \emph{converge} to $x$ if $d(x,x_n)\to0$ as $n\to\infty$. The sequence $\{x_n\}$
is said to \emph{converge to $x$ in the strong sense} if $\{x_n\}$ is Cauchy and
$\{x_n\}$ converges to $x$. The space $X$ is said to be \emph{complete} if
every Cauchy sequence in $X$ converges.

\begin{prop}[\cite{Suzuki-Alamri-Khan}]
\label{prop:d-is-continuous}
  Let $\{x_n\}$ and $\{y_n\}$ be sequences in $X$ that
  converge to $x$ and $y$ in the strong sense, respectively. Then
  \[
    d(x,y) = \lim_{n\to\infty} d(x_n,y_n).
  \]
\end{prop}

Branciari, in [1], proved a generalization of the Banach contraction principle.
As it is mentioned in \cite{Alamri-Suzuki-Khan}, their proof is not correct
because a $\nu$-generalized metric space does not necessarily have the compatible
topology; see \cite{Kadelburg-Radenovic-2}, \cite{Samet,Sarma,Suzuki}
and \cite{Turinici}.
A proof of the Banach contraction principle, as well as proofs of Kannan's and
\'Ciri\'c's fixed point theorems, in $\nu$-generalized metric spaces,
can be found in \cite{Suzuki-Alamri-Khan}.

\begin{thm}[\cite{Suzuki-Alamri-Khan}]
  Let $X$ be a complete $\nu$-generalized metric space, and let
  $T$ be a self-map of $X$. For every $x,y\in X$, let
  \begin{equation}
    m(x,y) = \max\set{d(x,y),d(x,Tx),d(y,Ty),d(x,Ty),d(y,Tx)}.
  \end{equation}
  Assume there exists $r\in[0,1)$ such that $d(Tx,Ty) \leq r m(x,y)$,
  for all $x,y\in X$. Then $T$ has a unique fixed point $z$ and, moreover,
  for any $x\in X$, the Picard iterates $T^n x$ $(n\in\N)$ converge to $z$ in the strong sense.
\end{thm}

The paper is organized as follows. In section \ref{sec:pre}, we study Cauchy sequences
in $\nu$-generalized metric spaces. We present a necessary and sufficient condition for
a sequence to be Cauchy. Next, in section \ref{sec:fixed-point-theorems}, we give new
fixed point theorems in $\nu$-generalized metric spaces. These results are generalizations
to $\nu$-generalized metric spaces of theorems of Meir and Keeler \cite{Meir-Keeler-1969},
\'Ciri\'c \cite{Ciric-1981} and Matkowski \cite[Theorem 1.5.1]{Kuczma}, and
Proinov \cite{Proinov-2006}.

\section{Results on Cauchy Sequences}
\label{sec:pre}

The following is the main result of the section.

\begin{lem}
\label{lem:technical-lemma}
  Let $\{x_n\}$ be a sequence in a $\nu$-generalized metric space $X$ such that
  $x_n\ (n\in\N)$ are all different. Suppose, for every $\e>0$, for any two subsequences
  $\{x_{p_i}\}$ and $\{x_{q_i}\}$, if\/ $\limsup\limits_{i\to\infty}d(x_{p_i},x_{q_i})\leq \e$,
  then, for some $N$,
  \begin{equation}\label{eqn:d(xpn+nun,xqn+nun)<=e}
    d(x_{p_i+1},x_{q_i+1}) \leq \e \quad (i\geq N).
  \end{equation}
  If $d(x_n,x_{n+1})\to0$, then $\{x_n\}$ is $\nu$-Cauchy.
\end{lem}

\begin{proof}
  Suppose $\set{x_n}$ is not $\nu$-Cauchy. Then \eqref{eqn:k-Cauchy} fails to hold
  for $k=\nu$. Hence, there is $\e>0$ such that
  \begin{equation}\label{eqn:negation-of-Cauchy}
    \forall k\in\N,\ \exists\, n\geq k,
    \quad \sup\set{d(x_n,x_{n+1+m\nu}):m\in\Z^+}>\e.
  \end{equation}

  Since $d(x_n,x_{n+1})\to0$, there exist positive integers $k_1<k_2<\dotsb$ such that
  \[
    d(x_n,x_{n+1}) < \e/i \quad (n\geq k_i).
  \]

  \noindent
  For each $k_i$, by \eqref{eqn:negation-of-Cauchy}, there exist
  $n_i\geq k_i+1$ and $m_i\in\Z^+$ such that
  \[
    d(x_{n_i},x_{n_i+1+m_i\nu})>\e.
  \]
  Since $d(x_{n_i},x_{n_i+1})<\e$, we have $m_i\geq 1$. We let $m_i$ be the smallest number
  with this property so that $d(x_{n_i},x_{n_i+1+m_i\nu-\nu}) \leq \e$.
  Now, let $p_i=n_i-1$ and $q_i=n_i+m_i\nu$. Then $q_i > p_i \geq k_i$, and
  \[
    d(x_{p_i+1},x_{q_i+1})>\e,\quad
    d(x_{p_i+1},x_{q_i+1-\nu}) \leq \e.
  \]

  Using property \eqref{item:nu-angle-inequality} in Definition \ref{dfn:nu-generalized-ms},
  since all $x_n\ (n\in\N)$ are different, for every $i\in\N$, we have
  \begin{equation*}
  \begin{split}
    d(x_{p_i},x_{q_i})
    \leq d(x_{p_i},x_{p_i+1})  & + d(x_{p_i+1},x_{q_i+1-\nu}) \\
     & + d(x_{q_i+1-\nu},x_{q_i-\nu}) + \dotsb + d(x_{q_i-1},x_{q_i}).
  \end{split}
  \end{equation*}

  \noindent
  Therefore, $d(x_{p_i},x_{q_i}) \leq \nu\e/i + \e$, and thus
  $\limsup\limits_{i\to\infty} d(x_{p_i},x_{q_i}) \leq \e$. This is a contradiction,
  since $d(x_{p_i+1},x_{q_i+1})>\e$, for all $i$.
\end{proof}

\begin{thm}
  Suppose $\{x_n\}$ satisfies all conditions in Lemma $\ref{lem:technical-lemma}$,
  and, moreover, $d(x_n,x_{n+2})\to0$. Then $\{x_n\}$ is Cauchy.
\end{thm}

\begin{proof}
  By Lemma \ref{lem:technical-lemma}, the sequence $\{x_n\}$ is $\nu$-Cauchy.
  Since $d(x_n,x_{n+2})\to0$, by Proposition \ref{prop:nu-Cauchy-is-Cauchy},
  the sequence $\{x_n\}$ is Cauchy.
\end{proof}

\begin{thm}
\label{thm:main}
  Let $\{x_n\}$ be a sequence in $X$ such that
  $x_n\ (n\in\N)$ are all different and $d(x_n,x_{n+1})+d(x_n,x_{n+2})\to0$.
  Assume $\m(x,y)$ is a nonnegative function on $X\times X$ such that,
  for any two subsequences $\{x_{p_i}\}$ and $\{x_{q_i}\}$,
  \begin{equation}\label{eqn:limsup m <= limsup d}
      \limsup_{i\to\infty} \m(x_{p_i},x_{q_i})
      \leq \limsup_{i\to\infty} d(x_{p_i},x_{q_i}).
  \end{equation}
  The following condition then implies that $\{x_n\}$ is Cauchy:
  for every $\e>0$, for any two subsequences $\{x_{p_i}\}$ and $\{x_{q_i}\}$,
  if $\limsup \m(x_{p_i},x_{q_i})\leq \e$, then, for some $N$,
    \begin{equation*}
      d(x_{p_i+1},x_{q_i+1}) \leq \e \quad (i\geq N).
    \end{equation*}
\end{thm}

\begin{proof}
  Follows directly from Lemma \ref{lem:technical-lemma} and Theorem \ref{thm:main}.
\end{proof}

\section{Fixed Point Theorems of \'Ciri\'c-Matkowski Type}
\label{sec:fixed-point-theorems}

Let $(X,d\,)$ be a $\nu$-generalized metric space.
A mapping $T:X\to X$ is said to be a \emph{\CM{} contraction} if
$d(Tx,Ty)<d(x,y)$, for every $x,y\in X$, with $x\neq y$, and,
for any $\e>0$, there exists $\de>0$ such that
  \begin{equation}\label{eqn:CM-contraction}
    \forall x,y\in X, \quad
     d(x, y) < \de+\e \Longrightarrow d(Tx,Ty)\leq \e.
  \end{equation}

\begin{lem}[{\cite[Lemma 3.1]{abtahi-FPT}}]
\label{lem:equiv-conditions-m-contractive-sequence}
 For a sequence $\set{x_n}$ in $X$ and a nonnegative function $\m(x,y)$ on $X\times X$,
 the following are equivalent:
  \begin{enumerate}[\upshape(i)]
    \item \label{item:it-is-m-contractive-sequence}
    for every $\e>0$, there exist $\de>0$ and $N\in\Z^+$ such that
    \begin{equation}\label{eqn:m-contractive-sequence}
      \forall p,q\geq N, \quad
      \m(x_p,x_q) < \e+\de \Longrightarrow d(x_{p+1},x_{q+1}) \leq \e.
    \end{equation}

    \item \label{item:m-contractive-sequence-in-term-of-pnqn}
    for every $\e>0$, for any two subsequences $\{x_{p_i}\}$ and $\{x_{q_i}\}$,
    if \newline $\limsup \m(x_{p_i},x_{q_i})\leq \e$ then,
    for some $N$, $d(x_{p_i+1},x_{q_i+1}) \leq \e\ (i\geq N)$.
  \end{enumerate}
\end{lem}

Now, suppose $T$ is a \CM{} contraction on $X$, take a point $x\in X$, and
set $x_n=T^nx$ ($n\in\N$). Then, for every $\e>0$, there exist $\de>0$ such that
$d(x_p,x_q) < \e+\de$ implies $d(x_{p+1},x_{q+1}) \leq \e$. By the above lemma,

\begin{lem}
\label{lem:the-same-or-different}
  Let $T:X\to X$ be a mapping. Suppose $d(T^nx,T^{n+1}x)\to0$, for some $x\in X$.
  Then, for some $k\in\N$, either the picard iterates $T^n x$ $(n\geq k)$ are all
  different or they are all the same.
\end{lem}

\begin{proof}
  Suppose $T^{k+m}x=T^kx$, for some $k,m\in\N$, and let $m$ be the smallest positive integer
  with this property. If $m=1$, that is $T^{k+1}x=T^kx$, then $T^nx=T^kx$, for $n\geq k$, and there is nothing
  to prove. If $m\geq2$, then every two successive element in the following sequence
  are different:
  \[
   T^kx,T^{k+1}x,\dotsc,T^{k+m-1}x,T^{k+m}x,T^{k+m+1}x,\dotsc
  \]
\end{proof}

\begin{thm}
\label{thm:m-contractive-T-produces-Cauchy}
  Let $T$ be a self-map of $X$ and $\m(x,y)$ be a nonnegative function on $X\times X$.
  Suppose, for some point $x\in X$, the following conditions hold:
  \begin{enumerate}[\upshape(i)]
    \item for any $\e>0$, there exist $\de>0$ and $N\in\Z^+$ such that
    \begin{equation}\label{eqn:m-contractive-orbits}
      \forall p,q\geq N, \quad
       \m(T^px, T^qx) < \de+\e \Longrightarrow d(T^{p+1} x, T^{q+1}x)\leq \e,
    \end{equation}

    \item condition \eqref{eqn:limsup m <= limsup d} holds for any two subsequences
    $\{T^{p_i}x\}$ and $\{T^{q_i}x\}$ of\/ $\{T^nx\}$,

    \item $d(T^nx,T^{n+1}x)+d(T^nx,T^{n+2}x)\to0$.
  \end{enumerate}
  Then $\set{T^nx}$ is a Cauchy sequence.
\end{thm}

\begin{proof}
  Using Lemma \ref{lem:equiv-conditions-m-contractive-sequence},
  condition \eqref{eqn:m-contractive-orbits} implies that, for every $\e>0$,
  for any two subsequences $\{T^{p_i}x\}$ and $\{T^{q_i}x\}$ of $\{T^nx\}$,
  if $\limsup \m(T^{p_i}x,T^{q_i}x)\leq \e$ then, for some $N$,
  $d(T^{p_i+1}x,T^{q_i+1}x) \leq \e$ $(i\geq N)$. By Lemma \ref{lem:the-same-or-different},
  the Picard iterates $T^nx$ are eventually all the same, in which case $\{T^nx\}$
  is obviously a Cauchy sequence, or they are all different. In the latter case,
  Theorem \ref{thm:main} shows that $\{T^nx\}$ is Cauchy.
\end{proof}

\begin{cor}
  Let $T$ be a \CM{} contraction on $X$. Then $T$ has a unique fixed point $z$, and,
  moreover, for any $x\in X$, the sequence $\{T^nx\}$ converges to $z$
  in the strong sense.
\end{cor}

\begin{proof}
  First, we show that $T$ has at most one fixed point. Suppose $Tz=z$ and
  $y\neq z$. Then $d(Ty,Tz)=d(Ty,z)<d(y,z)$. Hence $Ty\neq y$.

  Given $x\in X$, we consider the following two cases.
  \begin{enumerate}[\quad(a)]
    \item There exists $k,m\in\N$ such that $T^{k+m}x=T^kx$.
    \item $T^nx$ $(n\in\N)$ are all different.
  \end{enumerate}

  In case (a), where $T^{k+m}x=T^kx$, for some $k,m\in\N$, we let $m$
  be the smallest positive integer with this property. If $m=1$, that is
  $T^{k+1}x=T^kx$, then $T^nx=T^kx$, for $n\geq k$, and there is nothing
  to prove. If $m\geq2$, then every two successive element in the following sequence
  are different:
  \[
   T^kx,T^{k+1}x,\dotsc,T^{k+m-1}x,T^{k+m}x,T^{k+m+1}x,\dotsc
  \]

  \noindent Recall that $x\neq y$ implies $d(Tx,Ty)<d(x,y)$. Hence
  \begin{align*}
    d(T^kx,T^{k+1}x)
     & = d(T^{k+m}x,T^{k+m+1}x) < d(T^{k+m-1}x,T^{k+m}x) \\
     & < \dotsb < d(T^{k+1}x,T^{k+2}x) < d(T^kx,T^{k+1}x).
  \end{align*}

  \noindent This is absurd.

  In case (b), we let $x_n=T^nx$, and show that $d(x_n,x_{n+i})\to0$, for $i=1,2$.
  Since $x_n$ $(n\in\N)$ are all different, we have $d(x_{n+1},x_{n+i+1})<d(x_{n},x_{n+i})$,
  for every $n$, that is, the sequence $\e_n=d(x_n,x_{n+i})$ is decreasing and thus
  $\e_n \downarrow\e$ for some $\e\geq0$. If $\e>0$, there is $\de>0$ such that
  $\e_n = d(T^nx,T^{n+1}x)\leq \e+\de$ implies that
  $\e_{n+1} = d(T^{n+1}x,T^{n+2}x)\leq \e$. This is a contradiction since
  we have $\e<\e_n$, for all $n$. Hence, $d(x_n,x_{n+i})\to0$ $(i=1,2)$. Now,
  by Theorem \ref{thm:m-contractive-T-produces-Cauchy}, the sequence $\{T^nx\}$
  is Cauchy. Since $X$ is complete, $\{T^nx\}$ converges to some $z\in X$.
  By Proposition \ref{prop:d-is-continuous}, we have
  \[
    d(z,Tz) = \lim_{n\to\infty} d(T^nx,Tz) \leq
    \lim_{n\to\infty} d(T^{n-1}x,z) = 0.
  \]
  Hence $Tz=z$, i.e., $z$ is a fixed point of $T$.
\end{proof}

\begin{lem}
  Let $\{x_n\}$ be a sequence in a $\nu$-generalized metric space $X$
  such that $x_n$ $(n\in\N)$ are all different.
  If $d(x_n,x_{n+1})+d(x_{n+1},x_{n+2})\to0$, then
  \[
    d(x_n,x_{n+m})\to0, \quad (m\geq 3).
  \]
\end{lem}

\begin{dfn}
  A self-mapping $T$ of a $\nu$-generalized metric space $X$ is said to be
  \emph{sequentially continuous} if $\{Tx_n\}$ converges to $Tx$ whenever
  $\{x_n\}$ converges to $x$. The mapping $T$ is called
  \emph{asymptotically regular} if
  \[
    d(T^nx,T^{n+1}x)+d(T^nx,T^{n+2}x)\to0\quad (x\in X).
  \]
\end{dfn}

We are now in a position to state and prove a version of
Proinov's theorem, \cite[Theorem 4.2]{Proinov-2006}, for $\nu$-generalized metric spaces.

\begin{thm}
\label{thm:Generalized-Proinov}
  Let $X$ be a complete $\nu$-generalized metric space, and $T$ be a
  sequentially continuous and asymptotically regular self-map of $X$.
  For $\ga>0$, define $\m$ on $X\times X$ by
  \begin{equation}\label{eqn:m(x,y)-in-generalized-Proinov}
    m(x,y)=d(x,y)+\ga \bigl(d(x,Tx)+d(y,Ty)\bigr).
  \end{equation}
  Suppose $d(Tx,Ty)<\m(x,y)$, for every $x,y\in X$, with $x\neq y$, and,
  for any $\e>0$, there exist $\de>0$ and $N\in\N_0$ such that
  \begin{equation}\label{eqn:m-contractive-in-generalized-Proinov}
    \forall x,y\in X, \quad
     \m(T^Nx, T^Ny) < \de+\e \Longrightarrow d(T^{N+1}x,T^{N+1}y)\leq \e.
  \end{equation}
  Then $T$ has a unique fixed point $z$, and, for any $x\in X$, the Picard iterates
  $T^nx$ $(n\in\N)$ converge to $z$ in the strong sense.
\end{thm}

\begin{proof}
  First, let us prove that $T$ has at most one fixed point. If $Ty=y$ and $Tz=z$.
  Then $\m(y,z)=d(y,z)=d(Ty,Tz)$. Hence $y=z$.

  Now, choose $x\in X$ and set $x_n=T^nx$ $(n\in\N)$.
  Since $T$ is assumed to be asymptotically regular, we have
  $d(x_n,x_{n+1})\to0$. Hence, \eqref{eqn:limsup m <= limsup d}
  holds, for any two subsequences $\{x_{p_i}\}$ and $\{x_{q_i}\}$.
  By Theorem \ref{thm:main}, the sequence $\{T^nx\}$ is Cauchy and,
  since $X$ is complete, it converges to some point $z\in X$.
  Since $T$ is sequentially continuous, we have $Tz=z$.
\end{proof}

%


\begin{thebibliography}{99}
 \bibitem{abtahi-FPT}
  M. Abtahi,
  \textit{Fixed point theorems for Meir-Keeler type contractions in metric spaces},
  Fixed Point Theory (to appear).

 \bibitem{Alamri-Suzuki-Khan}
  B. Alamri, T. Suzuki and L.A. Khan,
  \textit{Caristi's Fixed Point Theorem and Subrahmanyam's Fixed Point Theorem
  in $\nu$-Generalized Metric Spaces},
  Journal of Function Spaces, 2015, Article ID 709391.

 \bibitem{Branciari}
  A. Branciari,
  \textit{A fixed point theorem of Banach-Caccioppoli type on a class
  of generalized metric spaces},
  Publicationes Mathematicae Debrecen, 57 (2000) 31--37.

\bibitem{Ciric-1981}
  Lj. B. \'Ciri\'c,
  \textit{A new fixed-point theorem for contractive mappings},
  Publ. Inst. Math. (N.S) \textbf{30} (44) (1981), 25--27.


 \bibitem{Kadelburg-Radenovic-1}
  Z. Kadelburg and S. Radenovi\'c,
  \textit{On generalized metric spaces: A survey},
  TWMS J. Pure Appl. Math., 5 (2014), 3--13.

\bibitem{Kadelburg-Radenovic-2}
 Z. Kadelburg and S. Radenovi\'c,
 \textit{Fixed point results in generalized metric spaces without Hausdorff property},
 Mathematical Sciences, vol. 8, article 125, 2014.

 \bibitem{Kannan}
  R. Kannan, \textit{Some results on fixed points-II},
  Amer. Math. Monthly, 76 (1969), 405--408.

 \bibitem{Kirk-Shahzad}
  W. A. Kirk and N. Shahzad,
  \textit{Generalized metrics and Caristi's theorem},
  Fixed Point Theory Appl., 2013, 2013:129.

\bibitem{Kuczma}
  M. Kuczma, B. Choczewski, R. Ger,
  \textit{Iterative functional equations, Encyclopedia of Mathematics and Applications},
  vol. 32, Cambridge University Press, Cambridge, 1990.

\bibitem{Meir-Keeler-1969}
  A. Meir, E. Keeler,
  \textit{A theorem on contraction mappings},
  J. Math. Anal. Appl., \textbf{28} (1969), 326--329.

\bibitem{Proinov-2006}
  Petko D. Proinov,
  \textit{Fixed point theorems in metric spaces},
  Nonlinear Anal. 64 (2006), 546--557.

\bibitem{Samet}
 B. Samet,
 \textit{Discussion on `a fixed point theorem of Banach-Caccioppoli type on a class
 of generalized metric spaces' by A. Branciari},
 Publicationes Mathematicae, vol. 76, no. 4, pp. 493--494, 2010.

\bibitem{Sarma}
 I.R. Sarma, J.M. Rao, S.S. Rao,
 \textit{Contractions over generalized metric spaces},
 Journal of Nonlinear Science and its Applications, vol. 2, no. 3, pp. 180--182, 2009.

 \bibitem{Suzuki}
  T. Suzuki,
  \textit{Generalized metric spaces do not have the compatible topology},
  Abstr. Appl. Anal., 2014, Art. ID 458098, 5 pp.

 \bibitem{Suzuki-Alamri-Khan}
  T Suzuki, B Alamri and L.A. Khan,
  \textit{Some notes on fixed point theorems in $\nu$-generalized metric spaces},
  Bull. Kyushu Inst. Tech. Pure Appl. Math. 62 (2015), 15--23.

\bibitem{Turinici}
 M. Turinici,
 \textit{Functional contractions in local Branciari metric spaces},
 ROMAI Journal, vol. 8, no. 2, pp. 189--2012.

\end{thebibliography}
\end{document}